\documentclass[reqno]{amsart}
\usepackage[english]{babel}
\usepackage{amssymb,verbatim}
\usepackage[T1,T5]{fontenc}

\newtheorem{theorem}{Theorem}[section]
\newtheorem{proposition}[theorem]{Proposition}
\newtheorem{lemma}[theorem]{Lemma}
\newtheorem{corollary}[theorem]{Corollary}

\theoremstyle{definition}
\newtheorem{definition}[theorem]{Definition}
\newtheorem{example}[theorem]{Example}

\theoremstyle{remark}
\newtheorem*{remark}{Remark}

\numberwithin{equation}{section}

\newcommand{\Hiep}{{\fontencoding{T5}\selectfont Ph\d{a}m Ho\`{a}ng Hi\d \ecircumflex{}p}}

\newcommand{\Nguyen}{{\fontencoding{T5}\selectfont Nguy\~\ecircumflex{}n}}

\newcommand{\Cau}{{\fontencoding{T5}\selectfont  136 Xu\^an Thu\'{y}, C\`\acircumflex{}u Gi\~ay, H\`{a} N\d\ocircumflex{}i\ }}

\newcommand{\CEP}[1]{\mbox{$\mathbb{C}^{#1}$}}
\newcommand{\C}{\mbox{$\mathbb{C}$}}

\newcommand{\B}{\mbox{$\mathbb{B}$}}

\newcommand{\E}{\mbox{$\mathcal{E}$}}
\newcommand{\Eo}{\mbox{$\mathcal{E}_{0}$}}

\newcommand{\F}{\mbox{$\mathcal{F}$}}

\newcommand{\PSH}[1]{\mbox{$\mathcal{PSH}(#1)$}}

\newcommand{\ddcn}[1]{\mbox{$\left(dd^{c}#1\right)^{n}$}}

\begin{document}

\author{Per \AA hag}
\address{Department of Mathematics and Mathematical Statistics\\ Ume\aa \ University\\ SE-901 87 Ume\aa \\ Sweden}
\email{Per.Ahag@math.umu.se}
\thanks{The first-named author was partially supported by the Lars Hierta Memorial Foundation.}
\author{Urban Cegrell}
\address{Department of Mathematics and Mathematical Statistics\\ Ume\aa \ University\\ SE-901 87 Ume\aa \\ Sweden}
\email{Urban.Cegrell@math.umu.se}
\author{\Hiep}
\address{Department of Mathematics\\
Hanoi National University of Education \\ \Cau \\
Vietnam} \email{phhiep\_vn@yahoo.com}

\title{Monge-Am\`ere measures on subvarieties}

\begin{abstract}
In this article we address the question whether
the complex Monge-Amp\`{e}re equation is solvable for measures with large
singular part. We prove that under some conditions there are no solution when the right-hand side
is carried by a smooth subvariety in $\CEP{n}$ of dimension $k<n$.
\end{abstract}

\maketitle

\begin{center}\Large\bf
\today
\end{center}

\section{Introduction}
 In this article we study the complex Monge-Amp\`{e}re equation
\begin{equation}\label{intro}
\ddcn{u}=\mu\,
\end{equation}
where $\mu$ is a given non-negative Radon measure and $\ddcn{\,\cdot\,}$ denotes the
complex Monge-Amp\`{e}re operator.  Monge-Amp\`{e}re techniques have an interesting history with applications ranging from algebraic and complex geometry to dynamics and theoretical physics (see e.g.~\cite{cegrell_hiep_etc,berman,demailly_hiep,Dinh_et al,Dinh_Nguyen,Donaldson}). For an historical account of the complex Monge-Amp\`{e}re operator we refer to~\cite{kiselman,krylov}.

\bigskip

  In the seminal article~\cite{bed_tay1}, by Bedford and Taylor it was proved that if $u$ is a continuous
  plurisubharmonic function defined on $\Omega\subset\CEP{n}$ , then the left-hand side $\ddcn{u}$ of the
 Monge-Amp\`{e}re equation can not charge on any subvariety in $\Omega$ of dimension $k<n$. On the other hand, they
 show that $\ddcn{u}$ can charge at a single point and that~(\ref{intro}) have (in this case) no unique solution. Several
 author have studied the case when $\mu$ is given by a single point mass or a finite sum of such (see e.g.~\cite{celik_poletsky,demailly_Z,klimek_green,lempert1,lempert2}). In~\cite{ahag_czyz_cegrell_hiep}, a measure $\mu$ was constructed that do not have any atoms and it is supported by a pluripolar set such that the equation~(\ref{intro}) have a solution with this given measure. Hence, there exists a measure $\mu$ with large singular part for which
 equation~(\ref{intro}) is solvable. The case when the measure $\mu$ vanishes on all pluripolar subsets of $\Omega$
 was completed in~\cite{cegrell_gdm} (see also~\cite{ahag_czyz_cegrell_hiep}).

  The growing use of complex Monge-Amp\`{e}re techniques in applications imply a growing demand on knowledge of~(\ref{intro}) with a large singular part of the given right-hand side (see e.g.~\cite{phong1,phong2}). Therefore, we address in this article the following question:

\bigskip

\noindent {\bf Aim:} \emph{Let $\Omega$ be a bounded hyperconvex domain in $\mathbb C^n$ and let $S$ be smooth subvariety in $\Omega$ of dimension $k<n$.  Assume that $\mu$ is a non-negative Radon measure (not identically zero) defined on $S$ with finite total mass. Do there exists a  plurisubharmonic function such that $(dd^c u)^n = \mu$? (with suitable interpretation of dimensions) }

\bigskip

 Now let $\Omega\subset\CEP{n}$ be a bounded hyperconvex domain, and let $\E(X)$ be the largest subset of non-positive
 plurisubharmonic functions $u$ defined on a complex manifold $X$ (e.g $X=\Omega$) for which $\ddcn{u}$ is a well-defined non-negative Radon measure. Furthermore, let $\F(\Omega)\subset\E(\Omega)$ be the subset with finite total mass and
 essential boundary values zero (see section~\ref{sec_prelm} for details). Our question in hand is of a purely local nature, and therefore we can without loss of generality assume that $S=\Delta^k\times\{0\}^{n-k}$,
 where $\Delta\subset\C$ is the unit disc. Furthermore, for our purpose it is sufficient to make the assumption that the total mass of $\mu$ in~(\ref{intro}) is finite.

 From Theorem~\ref{sufficient} it follows that if there exists a function  $\varphi\in\E(S)$ such that $(dd^c\varphi )^k ( \{\varphi >-\infty\} )=0$, then there exists a function $u\in\mathcal F(\Omega )$ such that
 \[
 (dd^c u)^n = \mu\times\delta_0^{n-k}
 \]
 with $\mu = (dd^c\varphi)^k$. Here $\delta_0$ denotes the dirac measure at the origin of $\Delta$. It should be emphasized that if $u\in\E(\Omega)$ and $u|_S$ is not identically $-\infty$, then by Theorem 5.11 in \cite{cegrell_gdm}, we have that $(dd^c u)^n (\{ S\backslash \{u=-\infty\} \})=0$, and therefore there exists a pluripolar Borel set $E$ in $S$ such that $(dd^cu)^n (S\backslash E) = 0$. Example \ref{example1} shows that this question is more involved. We construct a function $u\in\F (\Omega)$ such that $\ddcn{u}=\delta_{0}$, and $\overline{\{u=-\infty\}}=\Omega$. To show that the situation is even more intricate we construct in Example~\ref{example2} an example of a non-positive plurisubharmonic function $u$ with $u(z) > -\infty$ for all $z$,  but $\ddcn{u}$ is not a well-defined Radon measure.

  We end this article in section~\ref{toric} by proving the following: \emph{Assume that $\mu$ is a non-negative Radon measure defined on $\Delta^k$ with finite total mass such that it vanishes on all pluripolar sets in $\Delta^k$.
  Then there exists no function $u\in\E(\Delta^n)$ such that $u(z',z'') = u(z',|z_{k+1}|,...,|z_n|)$ and
  $(dd^c u)^n = \mu\times\delta_0^{n-k}$. Here we have that $z'=(z_1,...,z_k)$.}

  \bigskip

 For further information on pluripotential theory we refer to~\cite{ko1,ko2,ko3}

\section{Preliminaries}\label{sec_prelm}

Following the notation introduced by the second-named author in~\cite{cegrell_pc,cegrell_gdm} for a bounded hyperconvex domain $\Omega\subset\CEP{n}$ we define:
\[
\begin{aligned}
\Eo(\Omega)&=\bigg\{\varphi\in\PSH{\Omega}\cap L^\infty (\Omega): \lim_{z\to\partial\Omega}\varphi (z)=0, \int_\Omega \ddcn{\varphi}<\infty\bigg\},\\
\F(\Omega)&=\left\{\varphi\in\PSH{\Omega}: \exists \, \{u_j\}\subset \Eo(\Omega), \; \varphi_j\searrow \varphi, \, \sup_{j}\int_{\Omega} \ddcn{\varphi_{j}}<\infty\right\}\, ,\\
\E(\Omega)&=\big\{\varphi\in \PSH{\Omega}: \forall\, \omega\Subset \Omega \,\, \exists\, \varphi_{\omega}\in \F(\Omega) \text{ such that } \varphi_{\omega}=\varphi \text { on } \omega \big\}\, .
\end{aligned}
\]
We also need the following generalization to a complex manifold $X$:
\begin{multline*}
\E(X) = \Bigg\{u\in\PSH{X}:\ z\in X\text{ there exist a neighbourhood } W \text{ of } z\\ \text{ such that } u\in\E(W)\Bigg\}\, .
\end{multline*}

\bigskip

In the following example we show that there exists a function $u\in\F(\Omega)$ such that $\ddcn{u}=\delta_{0}$, and
$\overline{\{u=-\infty\}}=\Omega$.
\begin{example}\label{example1}  Let $\Omega\subset\CEP{n}$ be a hyperconvex domain in $\mathbb C^n$. This example shows
that there exists a function $u\in\mathcal F (\Omega)$ such that $\ddcn{u}=\delta_{0}$, and
$\overline{\{u=-\infty\}}=\Omega$.

\bigskip

 \emph{Step 1:} For $j\geq 1, 1\leq m\leq n$ let $\{a_{mj}\}_{j\geq 1}$, $a_{mj}>0$, be sequences of real numbers such that
\[
\sum_{j=1}^{\infty} \left( a_{1j}\cdots a_{nj} \right)^{\frac 1 n} < +\infty\,  ,
\]
and
\[
\sum_{j=1}^{\infty} \min ( a_{1j},\ldots,a_{m-1 j},a_{m+1 j},\ldots,a_{nj} )= \infty \qquad \text{ for all } 1\leq m\leq n\, .
\]
To simplify the notation let $A(j)=\min (a_{1j},\ldots,a_{m-1 j},a_{m+1 j},\ldots,a_{nj})$.
Set
\[
u(z)=\sum_{j=1}^{\infty}\max\left(a_{1j}\ln|z_1|,\ldots,a_{nj}\ln|z_n|\right)\, .
\]
Then we have that $u\in\F(\Delta^n)$, and $(dd^c u)^n=c\delta_{0}$ for some
\[
c\in \left[ \sum_{j=1}^{\infty} a_{1j}\cdots a_{nj}\, ,\,  \left(\left(\sum_{j=1}^{\infty} ( a_{1j}\cdots a_{nj}
\right)^{\frac 1 n }\right)^n \right]\, .
\]
Furthermore, we have that
\begin{multline*}
u(z_1,\ldots,z_{m-1},0,z_{m+1},\ldots,z_n)\\
\leq \sum_{j=1}^{\infty} A(j)\max ( \log |z_1|,\ldots,\log |z_{m-1}|,\log |z_{m+1}|,\ldots,\log |z_n| ) =-\infty\, .
\end{multline*}
Hence,
\[
\{u=-\infty\}=\{0\}\times\Delta^{n-1}\cup \cdots \cup \Delta^{n-1}\times \{0\}\, .
\]

\bigskip

 \emph{Step 2:} We can assume that the unit ball $\B$ is contained in $\Omega$. Let $\{S_j\}$ be a family of hyperplanes such that $\overline{\bigcup_{j=1}^{\infty} (S_j\cap\B)}=\overline\B$. By using \emph{step 1} together with changing coordinates we can choose $\varphi_j\in\F(\B)$ such that
\[
\ddcn{\varphi_j}=\frac{1}{2^j}\,\delta_0\qquad\text{ and }\qquad \varphi_j|_{S_j\cap\mathbb{B}}=-\infty\, .
\]
Set
$$\psi=\sum_{j=1}^{\infty}\varphi_j.$$
Then $\psi\in\mathcal F (\B)$,
$\psi|_{S_j\cap\mathbb{B}}=-\infty$ for all $j$, and $\ddcn{\psi}\geq \delta_{0}$. Set
\[
\psi^r=\sup\{\Phi\in\PSH{\B}: \Phi\leq 0 \text{ and } \Phi\leq \psi \text{ on } \B (0,r)\}\, .
\]
Here $\B (0,r)\subset\CEP{n}$ is the ball with radius $r$. This construction yields that $\{\psi^r\}$ increases
pointwise to a function $\varphi\in\F (\B)$ and $\ddcn{\varphi}=c\delta_{0}$, $c>0$. From the fact that $\psi^r\leq \varphi_j$ on $\B (0,r)$ and $(dd^c\varphi )^n=0$ on $\B\backslash \{ 0 \}$, we get that $\varphi\leq\varphi_j$ on $\B$ for all $j\geq 1$, which yields that
$\varphi|_{S_j\cap\mathbb{B}}=-\infty$ for all $j\geq 1$. Finally, set
\[
u=\sup\{ v\in\PSH{\Omega }: v\leq 0 \text{ and } v\leq \varphi \text{ on } \B \}\, .
\]
By Lemma~4.5 in \cite{hiep_sub}, Theorem~2.2 in \cite{cegrell_Lisa} and Lemma 4.1 in \cite{ahag_czyz_cegrell_hiep}, we get $u\in\mathcal F(\Omega)$, $(dd^c u)^n = c\delta_{0}$ and $\overline{\{u=-\infty\}}=\Omega$.

\end{example}

\bigskip

\begin{proposition}\label{case2_lem1} Assume that $\Omega\subseteq\CEP{n}$ is a bounded hyperconvex domain.  Let $u\in\F (\Omega)$ and $v\in\PSH{\Omega}$, $v\leq 0$, and $w\in \E(\Omega)$ be such that $\ddcn{w}$ vanishes on pluripolar sets. If $\ddcn{u} ( \{u>-\infty\} )=0$ and $u\geq v + w$ on a neighborhood $D$ of  $\{u=-\infty\}$ then $u\geq v$ on $\Omega$.
\end{proposition}
\begin{proof} We have that
\[
\max(u,v)+ w=\max(u+w,v+w)\leq u\, ,
\]
and therefore by Lemma~4.1 in~\cite{ahag_czyz_cegrell_hiep} we get that
\[
\ddcn{\max(u,v)}\geq \chi_{\{u=-\infty\}}\ddcn{u}=\ddcn{u}\, .
\]
Therefore, by Proposition~3.4 in~\cite{Khue_Hiep} implies that
$u=\max(u,v)\geq v$ on $\Omega$.
\end{proof}

\section{A sufficient condition on $\mu$}

\begin{lemma}\label{product} Assume that $\Omega_1\subset C^{n_1}$ and $\Omega_2\subset C^{n_2}$ are bounded hyperconvex domains. Let $u_1\in\E(\Omega_1)$, $u_2\in\E(\Omega_2)$ be such that
\[
\ddcn{u_1} ( \{u_1>-\infty\} )=\ddcn{u_2} ( \{u_2>-\infty\} )=0\, .
\]
Then
\begin{equation}\label{product1}
(dd^c\max (u_1,u_2))^{n_1+n_2} = (dd^c u_1)^{n_1}\wedge (dd^c u_2)^{n_2}\, .
\end{equation}
\end{lemma}
\begin{proof}
Set $u_1^j = \max (u_1,-j)$ and $u_2^j = \max (u_2,-j)$. From~\cite{blocki} (see also~\cite{ahag_cegrell_hiep}),
we have that
\[
(dd^c\max (u_1^j,u_2^j))^{n_1+n_2} = (dd^c u_1^j)^{n_1}\wedge (dd^c u_2^j)^{n_2}\, .
\]
By letting $j\to\infty$, we obtain that~(\ref{product1}).
\end{proof}
\begin{lemma}\label{productgreen} Let $\varphi\in\mathcal E(\Delta^k)$ be such that $(dd^c\varphi )^k ( \{\varphi>-\infty\} )=0$. Then
\[
(dd^c \max (\varphi(z_1,...,z_k),\log |z_{k+1}|, ...,|z_n|))^n = (dd^c\varphi)^k\times \delta_0^{n-k}\, .
\]
\end{lemma}
\begin{proof}
It follows from Lemma \ref{product}.
\end{proof}

From Lemma~\ref{productgreen} we have that

\begin{theorem}\label{sufficient}  Let $\Omega$ be a bounded hyperconvex domain in $\mathbb C^n$ and $S$ be a subvariety in $\Omega$ with dimension $k<n$. Assume that $\varphi\in\mathcal E(S)$ such that
\[
(dd^c\varphi )^k ( \{\varphi >-\infty\} )=0\, .
\]
Then exists a function $u\in\mathcal E(\Omega )$ such that $(dd^c u)^n = (dd^c\varphi)^k$.
\end{theorem}

\section{A necessary condition to belong to $\mathcal E(\Omega)$}\label{sec_necessary}

In this section we start with introducing some notation.
For $z=(z_1,...,z_n)\in\mathbb C^n$, we write $z'=(z_1,...,z_k)$ and $z"=(z_{k+1},...,z_n)$. Then we define
\begin{align*}
\|z\| & =\max(|z_1|,\ldots,|z_n|)\, ,\\
\|z'\| &=\max(|z_1|,\ldots,|z_k|)\, ,\\
\|z''\| &=\max(|z_{k+1}|,\ldots,|z_n|), .
\end{align*}

With these notation we make the following definition

\begin{definition}\label{new function} Let $u\in\PSH{\Delta^n}, u\leq 0$. We define
\[
\phi_u(z',r) = \frac { \max_{\|z''\|=r} u(z',z'') } { |\log r| }\, ,
\]
and
\[
\phi_u(z')=\left(-\nu_u(z',\cdot) (0)\right)^*\, ,
\]
where $\nu_u(z',\cdot) (0)$ is the Lelong number of the function $u(z',\cdot)$ at $0$.
\end{definition}

From the construction in Definition~\ref{new function} we get that $\phi_u(\cdot,r)\in\PSH{\Delta^k}$,
$\phi_u(\cdot,r)\leq 0$ and that $\phi_u (z',r)\nearrow -\nu_u(z',.) (0)$ as $r\searrow 0$. Thanks to~\cite{bed_tay2}, we have that $\phi_u\in\PSH{\Delta^k}$, $\phi_u\leq 0$, and that the set
\[
\left\{ z'\in\Delta^k:\ \phi_u(z') \not = -\nu_u(z',.) (0) \right\}
\]
is a pluripolar set in $\Delta^k$. Furthermore, we get that

\begin{itemize}\itemsep2mm

\item if $u\geq v$, then $\phi_u\geq\phi_v$

\item $\phi_{au+bv} = a\phi_u+b\phi_v$, for all $u,v\in\PSH{\Delta^n}$, $u,v\leq 0$, and $a,b\geq 0$

\item $\phi_{ \max (u,v) }= \max (\phi_u,\phi_v)$\, .

\end{itemize}

\bigskip

\begin{theorem}\label{constantheorem} Let $u\in\PSH{\Delta^n}$, $u\leq 0$. Then we have that $\phi_u$ is a constant function.
\end{theorem}
\begin{proof}
Take $z_0'\in\Delta^k$. We will only need to prove that
\[
\phi_u (z') \leq \phi_u(z_0') \qquad \text{ for all } \; z'\in\Delta^k\, .
\]
Fix $\epsilon >0$. We can choose $r>0$ small enough such that
\[
\phi_u(z')\leq \phi_u(z_0')+\epsilon \qquad \text{ for all } \; \|z'-z_0'\|<r\, .
\]
This implies that
\[
\nu_u(z',\cdot) (0)\geq -\phi_u(z_0')-\epsilon \qquad \text{ for all } \;\|z'-z_0'\|<r\, .
\]
Therefore, we have that
\[
u(z',z'')\leq (-\phi_u(z_0')-\epsilon)\log \|z''\| \qquad \text{ for } \|z'-z_0'\|<r,\; z''\in\Delta^{n-k}\, .
\]
Hence,
\[
 \{z'\in\Delta^k:\ |z'-z_0'|<r\}\times\{0\}^{n-k}\subset \{ z\in\Delta^n:\ \nu_u (z)\geq -\phi_u(z_0')-\epsilon \}\, .
\]
On the other hand, from Siu's theorem (see e.g.\cite{Siu74,De87}) we have that
\[
\{z\in\Delta^n:\ \nu_u (z)\geq -\phi_u(z_0')-\epsilon\}
\]
is an analytic set, which implies that
\[
\{z\in\Delta^n:\ \nu_u (z)\geq -\phi_u(z_0')-\epsilon\}=\Delta^k\times\{0\}^{n-k}\, .
\]
Thus,
\[
u(z',z'')\leq (-\phi_u(z_0')-\epsilon)\log \|z''\| \qquad \text{ for all } z\in\Delta^n\, .
\]
Hence, $\phi_u (z') \leq \phi_u(z_0')+\epsilon$ for all $z'\in\Delta^k$. Let now $\epsilon \to 0^+$, and we finally
get that
\[
\phi_u (z') \leq \phi_u(z_0')\qquad \text{ for all } z'\in\Delta^k\, .
\]
\end{proof}
\begin{remark} If $k=n-1$, then $\left(u-\phi_u \log \|z''\|\right)\in\PSH{\Delta^n}$.
\end{remark}
\begin{lemma}\label{lem1} Let $u$ be a pluriharmonic function, and let $\{u_j\}$ be a sequence of plurisubharmonic functions that converges to $u$ in $dV_{2n}$ on $\Omega$ as $j\to \infty$. Then $\{u_j\}$ converges to $u$ in capacity, as $j\to \infty$.
\end{lemma}
\begin{proof} Let $K\Subset L\Subset D\Subset \Omega$, and $\delta>0$. We shall prove that
\[
\text{Cap}_D(\{|u_j-u|>\delta\}\cap K)\to 0,\text{ as } j\to +\infty\, ,
\]
Choose $\phi\in\Eo(D)$ that satisfies $\ddcn{\phi}=dV_{2n}$. Take $A>0$ such that $A\phi\leq -1$ on $L$. Let
$0<\varepsilon<\frac{\delta}{2}$. Hartog's theorem yields that there exists a $j_0$ such that
\[
u_j\leq u+\varepsilon \text{ for all } z\in D, \text{ and } j\geq j_0\, .
\]
By Lemma 3.3 in \cite{ahag_czyz_cegrell_hiep},we have that
\begin{multline*}
\text{Cap}_D(\{|u_j-u|>\delta\}\cap K)=\text{Cap}_D(\{u-u_j>\delta\}\cap K)= \text{Cap}_D(\{u_j<u-\delta\}\cap K)\\
=\sup\left\{\int_{\{u_j<u-\delta\}\cap K}\ddcn{\varphi}:\varphi\in\PSH{D},\; -1\leq\varphi\leq 0\right\}\\
=\sup\left\{\int_{\{u_j<u-\delta\}\cap K}\ddcn{\varphi}:\varphi\in\PSH{D},\; h_{D,L}^*\leq\varphi\leq 0\right\}\\
\leq \frac1{\delta}\sup\left\{\int_D(u-u_j+\varepsilon)\ddcn{\varphi}:\varphi\in\PSH{D},\; h_{D,L}^*\leq\varphi\leq 0\right\}\\
\leq \frac1{\delta}\sup\left\{\int_D(u-u_j+\varepsilon)\ddcn{\varphi}:\varphi\in\PSH{D},\; A\phi\leq\varphi\leq 0\right\}\\
\leq \frac1{\delta}\int_D(u-u_j+\varepsilon)\ddcn{A\phi}=\frac{A^n}{\delta}\int_D(u-u_j+\varepsilon)\ddcn{\phi}\\=
\frac{A^n}{\delta}\int_D(u-u_j+\varepsilon)\, dV_{2n}
\leq \frac{A^n}{\delta}\left(\int_D|u-u_j| dV+\varepsilon V_{2n} (D)\right)\, .
\end{multline*}
Hence,
\[
\limsup_{j\to +\infty}\,\text{Cap}_D(\{|u_j-u|>\delta\}\cap K)\leq \varepsilon\frac{A^n V_{2n} (D)}{\delta}\quad \text{ for all } \varepsilon>0\, .
\]
Thus,
\[
\text{Cap}_D\left(\left\{|u_j-u|>\delta\right\}\cap K\right)\to 0,\text{ as } j\to +\infty\, ,
\]
\end{proof}
\begin{theorem}\label{capacity} Let $u\in\PSH{\Delta^n}$, $u\leq 0$. Then we have that
\[
\frac { u(z',rz'') } { |\log r |} \to \phi_u
\]
in capacity on $\Delta^k\times \left(\Delta_{\frac 1 r}\right)^{n-k}$, as $r\to 0^+$. Here $\Delta_{\frac 1 r}\subseteq \C$ denotes the disc of radius $\frac 1 r$.
\end{theorem}
\begin{proof}
From
\[
\phi_u(z',r)=\max_{ \|z''\|=1 }\frac { u(z',rz'') } { |\log r |} \nearrow \phi_u
\]
as $r\to 0^+$, and Theorem 3.2.12 in \cite{Ho} we get that
\[
\frac { u(z',rz'') } { |\log r |} \to \phi_u \quad \text{ in } dV_{2n} \;\text{ on } \Delta^k\times (\Delta_{\frac 1 r})^{n-k}
\]
as $r\to 0^+$. We complete this proof by using Lemma~\ref{lem1} and obtain that
\[
\frac { u(z',rz'') } { |\log r |} \to \phi_u
\]
in capacity on $\Delta^k\times (\Delta_{\frac 1 r})^{n-k}$ as $r\to 0^+$.
\end{proof}
\begin{theorem}\label{classE} Let $u\in\mathcal E (\Delta^n)$. Then $\phi_u$ is identically $0$.
\end{theorem}
\begin{proof}
Assume that $\phi_u<0$. Hence
\[
u(z',z'')\leq - \phi_u \log \|z''\| \qquad \text{ on } \Delta^n\, .
\]
Hence, $\nu_u (z)\geq -\phi_u$ on $\Delta^k\times\{0\}^k$. This is not possible, since $u\in\mathcal E(\Delta^n)$.
\end{proof}

Example~\ref{example2} shows that the converse of Theorem~\ref{classE} is in generally false.

\begin{example}\label{example2}
 In this example we construct a function $u\in\PSH{\Omega}$, $u\leq 0$ such that
\[
 u(z) > -\infty \qquad \text{ for all } z\in\Omega\,
\]
but $u\not\in\E(\Omega)$. We can assume that $\Omega\subset\Delta^n$. Let $u$ be defined on $\Delta^n$ as
\[
u(z)=\sum_{j=1}^\infty \max\left(\frac 1 {2^j}\log \frac {|z_1-\frac 1 {2^j}|}{|1-\frac {z_1}{2^j}|},\frac {2^j} {j} \log |z_2|,\log |z_3|,\ldots,\log |z_n|, -2^j\right)\, .
\]
We start by proving that $u(z) >-\infty$ for all $z\in\Delta^n$. If $z_1=0$, then we have that
\[
u(0)\geq\sum_{j=1}^\infty\frac 1 {2^j}\log \frac 1 {2^j}>-\infty\, .
\]
If $z_1\not =0$ we choose $j_0$ be such that $|z_1|>\frac 1 {2^{j_0-1}}$. Hence
\[
u(z)\geq\sum_{j=1}^{j_0}-2^j+\sum_{j=j_0+1}^\infty\frac 1 {2^j}\log \frac {|z_1-\frac 1 {2^j}|}{|1-\frac {z_1}{2^j}|}\geq\sum_{j=1}^{j_0}-2^j+\log \frac {|z_1|}{4}\sum_{j=j_0+1}^\infty\frac 1 {2^j}>-\infty\, .
\]
Next, we shall show that $u\not\in\mathcal E(W)$ for all neighbourhoods $W$ of $0$. Set
\[
u_k=\sum_{j=1}^k \max\left(\frac 1 {2^j}\log \frac {|z_1-\frac 1 {2^j}|}{|1-\frac {z_1}{2^j}|},\frac {2^j} {j} \log |z_2|, \log |z_3|,\ldots,\log |z_n|, -2^j\right)
\]
We have $u_k\in\mathcal E_0 (\Delta^n )$ and $\varphi_k\searrow u$ as $k\to\infty$, and
\begin{multline*}
(dd^c u_k )^n\\ \geq \sum_{j=1}^k \left( dd^c \max\left(\frac 1 {2^j}\log \frac {|z_1-\frac 1 {2^j}|}{|1-\frac {z_1}{2^j}|},\frac {2^j} {j} \log |z_2|, \log |z_3|,...,\log |z_n|, -2^j\right) \right)^n
\\=\sum_{j=1}^k\frac 1 j \sigma_{\{ |z_1-\frac 1 {2^j}|=e^{-4^j} \}}\times \sigma_{\{ |z_2|=e^{-j} \}}\times \sigma_{\{ |z_3|=e^{-2^j} \}}...\times \sigma_{\{ |z_n|=e^{-2^j} \}}\, ,
\end{multline*}
where $\sigma_{\{ |z_j|=r \}}$ is the normalized surface measure on $\{ |z_j|=r \}$. Hence,
\[
\int_W (dd^c u_k )^n\to\infty
\]
as $k\to\infty$ for all neighbourhood $W$ of $0$.
\end{example}
\begin{definition} For each $u\in\PSH{\Omega_1\times\Omega_2}$ and $w_2\in\Omega_2$ we define
\begin{multline*}
E(u,t,w_2)= \left\{ z_1\in\Omega_1: u(z_1,z_2)\leq t\log \|z_2-w_2\|+O(1),\text{ for every } z_2\in\Omega_2 \right\}\\
=\{ z_1\in\Omega_1: \nu_{u(z_1,.)}(w_2)\geq t \}\, .
\end{multline*}
\end{definition}
\begin{theorem} Let $u\in\E (\Omega_1\times\Omega_2)$. Then
\[
\bigcup_{t > 0} E(u,t,w_2)
\]
is a pluripolar set in $\Omega_1$ for all $w_2\in\Omega_2$.
\end{theorem}
\begin{proof} Since this problem is purely local we can without loss of generality assume that $\Omega_1 = \Delta^k$, $\Omega_2 = \Delta^{n-k}$ and $w_2=0$. Theorem \ref{classE} yields that $\phi_u\equiv 0$. We have that
\[
\bigcup_{t > 0} E(u,t,0)=\{ z'\in\Delta^k:\ \phi_u(z') \not = -\nu_u(z',.) (0) \}\, ,
\]
and therefore it follows that $\bigcup_{t > 0} E(u,t,0)$ is a pluripolar set in $\Delta^k$.
\end{proof}

\section{The toric case}\label{toric}

\begin{theorem}\label{torictheorem} Let $u\in\E(\Delta^n)$ be such that $u(z',z'') = u(z',|z_{k+1}|,...,|z_n|)$.
Then there exists a Borel pluripolar set $E$ in $\Delta^k$ such that
\[
(dd^c u)^n ( (\Delta^k\backslash E) \times\{0\}^{n-k} ) = 0\, .
\]
\end{theorem}
\begin{proof}
Without loss generality we can assume that $u\in\F(\Delta^n)$. Theorem 6.3 in \cite{cegrell_gdm} yields that there exists a function $\varphi\in\Eo(\Delta^k)$, $0\leq f\in L^1((dd^c\varphi)^n)$, a non-negative Radon measure $\nu$ defined on $\Delta^k$, and a Borel pluripolar set $E\subset\Delta^k$ such that
\[
1_{ \Delta^k\times \{0\}^{n-k} }(dd^c u)^n = f(dd^c\varphi )^k+\nu\, ,
\]
and $\nu (\Delta^k\backslash E) = 0$. We shall prove that
\[
f(dd^c\varphi )^k=0\, .
\]
Fix $t\in (0,1)$. Thanks to Lemma 4.3 in \cite{ahag_czyz_cegrell_hiep}, we can find a function $v\in \F(\Delta^n)$ such that $v\geq u$, $(dd^c v)^n = 1_{ \Delta_t^n } f (dd^c\varphi )^k$ and
\[
v(z',z'') = v(z',|z_{k+1}|,\ldots,|z_n|)\, .
\]
Next choose a sequence  $\{r_j\}$ with $r_j\searrow 0$. By the quasicontinuity of $\phi_v (\cdot,r_j)$ (see e.g. \cite{bed_tay2}), we can find a decreasing sequence of open sets $\{G_m\}_{m\geq 1}$ in $\Delta_t^k$ such that
\[
\text{Cap}_{ \Delta^k } (G_m)<\frac 1 m\quad \text{and}\quad \phi_v (.,r_j)|_{\overline\Delta_t^k\backslash G_m}\text{ are continuous }.
\]
Furthermore, it can be chosen such that each element is continuous on $\overline\Delta_t^k\backslash G_m$ for all $j,m\geq 1$, and $\nu_{v(z',.)}(0) = 0$ on $\overline\Delta_t^k\backslash G_m$ for all $m\geq 1$. By Dini's theorem we have
$\phi_v (z',r_j)$ converges uniformly $0$ on $z'\in\overline\Delta_t^k\backslash G_m$, as $j\to\infty$, for all $m\geq 1$. Hence, for each $m$ we can choose $j_m$ such that
\[
\epsilon_m = -\min_{z'\in \overline\Delta_t^k\backslash G_m} \phi_v(z',r_{j_m})\searrow 0,\quad\text{ as } m\to\infty\, .
\]
Since $v(z',z'') = v(z',|z_{k+1}|,\ldots,|z_n|)$, we have that
\[
v(z',z'')\geq \epsilon_{j_m} (\log|z_{k+1}|+\ldots+\log|z_n|)\, ,
\]
for all $(z',z'')\in( \overline \Delta_t^k\backslash G_m )\times \overline\Delta_{r_{j_m}}^{n-k}$. Set
\[
w_m=\max (v,\epsilon_{j_m} (\log|z_{k+1}|+\ldots+\log|z_n|))\, ,
\]
and choose $v_l\in\mathcal E_0\cap C(\Delta^n)$ such that $v_l\searrow v$. Set
\[
h_{G_m,\Delta^k} = \sup\left\{\varphi\in\PSH{\Delta^k}: \varphi\leq -1\text{ on } G_m\right\}
\]
We have that
\begin{multline*}
\int_{ \overline\Delta_t^k\backslash G_m\times \overline{ \Delta }_{r_{j_m}}^{n-k} } ( dd^c w_m )^n\\ \geq \varlimsup\limits_{l\to\infty}\int_{ \overline\Delta_t^k\backslash G_m\times \overline{ \Delta }_{r_{j_m}}^{n-k} } \left( dd^c\max ( v_l,\epsilon_{j_m} (\log|z_{k+1}|+\ldots+\log|z_n|) -\frac 1 l ) \right)^n\, .
\end{multline*}
Since,
\[
\overline\Delta_t^k\backslash G_m\times \overline{ \Delta }_{r_{j_m}}^{n-k}\subset \left\{ v_l>\epsilon_{j_m} (\log|z_{k+1}|+\ldots+\log|z_n|) -\frac 1 l \right\}
\]
and $h(dd^c v_l)^n\to h(dd^c v)^n$ weakly as $l\to\infty$ for all $h\in\PSH{\Delta^n}\cap L^\infty(\Delta^n)$, we have
that
\begin{multline*}
\int_{ \overline\Delta_t^k\backslash G_m\times \overline{ \Delta }_{r_{j_m}}^{n-k} } ( dd^c w_m )^n\geq \varlimsup\limits_{l\to\infty}\int_{ \overline\Delta_t^k\backslash G_m\times \overline{ \Delta }_{r_{j_m}}^{n-k} } ( dd^c v_l )^n\\
\geq\varliminf\limits_{l\to\infty}\int_{ \Delta_t^k\times\Delta_{ r_{j_m} }^{ n-k } } ( 1 + h_{ G_{j_m},\Delta^k } ) ( dd^c v_l )^n
\geq\int_{ \Delta_t^k\times\Delta_{ r_{j_m} }^{n-k} } ( 1 + h_{ G_{j_m},\Delta^k } ) ( dd^c v )^n\\
=\int_{ \Delta_t^k} ( 1 + h_{ G_{j_m},\Delta^k } ) f ( dd^c \varphi )^k\, .
\end{multline*}
From the fact that
\[
\int_{ \Delta^k } (dd^c h_{ G_{j_m},\Delta^k } )^n = \text{Cap}_{ \Delta^k } ( G_{j_m} ) \searrow 0 \quad \text{ as } m\to\infty
\]
we get $h_{ G_{j_m},\Delta^k }\nearrow 0$ a.e on $\Delta^k$, as $m\to\infty$. This yields that
\[
\varliminf_{m\to\infty} \int_{ \overline\Delta_t^k\backslash G_m\times \overline{ \Delta }_{r_{j_m}}^{n-k} } ( dd^c w_m )^n\geq \int_{ \Delta_t^k} f ( dd^c \varphi )^k\, .
\]
On the other hand, since $v\leq w_m\nearrow 0$ as $m\to\infty$, we get that
\[
\varlimsup_{m\to\infty} \int_{ \overline\Delta_t^k \times \overline{ \Delta }_{r_{j_m}}^{n-k} } ( dd^c w_m )^n\leq\varlimsup_{m\to\infty} \int_{ \overline\Delta_t^k \times \overline{ \Delta }_{r_{j_1}}^{n-k} } ( dd^c w_m )^n\leq 0
\]
Thus,
\[
\int_{ \Delta_t^k} f ( dd^c \varphi )^k = 0
\]
To complete this proof let $t\to 1^{-}$.
\end{proof}

By combining Theorem~\ref{torictheorem} with Theorem~\ref{sufficient} we get the following corollary
\begin{corollary} \label{toricumea} Let $\mu$ be a non-negative Radon measure defined on $\Delta^k$ which vanish on every pluripolar sets in $\Delta^k$. Then there is exists no function $u\in\E(\Delta^n)$ such that
\[
u(z',z'') = u(z',|z_{k+1}|,\ldots,|z_n|)\quad\text{ and }\quad (dd^c u)^n = \mu\, .
\]
\end{corollary}

\end{document}